\documentclass[11pt]{article}
\usepackage{graphicx}
\usepackage{pgfplots}

\usepackage[colorlinks=true,citecolor=blue]{hyperref}
\usepackage{natbib}
\usepackage{graphicx}
\usepackage{amsfonts}
\usepackage{amsmath}
\usepackage{amssymb}
\usepackage{url}
\usepackage{fancyhdr}
\usepackage{indentfirst}
\usepackage{enumerate}
\usepackage{titlesec}
\usepackage{amsthm}
\usepackage{dsfont}
\usepackage[misc]{ifsym}

\theoremstyle{definition}
\newtheorem{example}{Example}

\theoremstyle{plain}
\newtheorem{theorem}{Theorem}

\newtheorem{proposition}{Proposition}

\theoremstyle{remark}
\newtheorem{remark}{Remark}

\usepackage{cases}

\theoremstyle{definition}

\def\N{\mathbb{N}}

\def\p{\mathbb{P}}
\def\E{\mathbb{E}}

\def\R{\mathbb{R}}

\DeclareMathOperator{\sign}{sgn}

\def\d{\mathrm{d}}

\renewcommand{\(}{\left(}
\renewcommand{\)}{\right)}
\renewcommand{\[}{\left[}
\renewcommand{\]}{\right]}

\usepackage[onehalfspacing]{setspace}

\usepackage{bm}
\usepackage{tikz-qtree}
\usepackage{tikz}

\def\id{\mathds{1}}

\usepackage{tikz}

\usepackage{caption}

\usetikzlibrary{decorations.pathreplacing}

\usetikzlibrary{cd} 
\usepackage{tkz-graph}
\usepackage{pgfplots}
\usepackage{subcaption}
\usetikzlibrary{calc}

\usepgfplotslibrary{fillbetween}

\usetikzlibrary{patterns, positioning, arrows}
\usetikzlibrary{arrows.meta}

\usepackage{schemabloc}
\usetikzlibrary{decorations.markings}

 \pgfdeclarepatternformonly{south east lines}{\pgfqpoint{-0pt}{-0pt}}{\pgfqpoint{3pt}{3pt}}{\pgfqpoint{3pt}{3pt}}{
    \pgfsetlinewidth{0.4pt}
    \pgfpathmoveto{\pgfqpoint{0pt}{3pt}}
    \pgfpathlineto{\pgfqpoint{3pt}{0pt}}
    \pgfpathmoveto{\pgfqpoint{.2pt}{-.2pt}}
    \pgfpathlineto{\pgfqpoint{-.2pt}{.2pt}}
    \pgfpathmoveto{\pgfqpoint{3.2pt}{2.8pt}}
    \pgfpathlineto{\pgfqpoint{2.8pt}{3.2pt}}
    \pgfusepath{stroke}}

\title{Sub-uniformity of harmonic mean p-values}
\author{Yuyu Chen\thanks{Department of Economics, University of Melbourne,  Australia. \Letter~{\scriptsize\url{yuyu.chen@unimelb.edu.au}}}
\and  Ruodu Wang\thanks{Department of Statistics and Actuarial Science, University of Waterloo,  Canada. \Letter~{\scriptsize\url{wang@uwaterloo.ca}}}
\and Yuming Wang\thanks{School of Mathematical Sciences, Peking University, China. \Letter~{\scriptsize\url{wangyuming@pku.edu.cn}} 
}
\and Wenhao Zhu\thanks{School of Mathematical Sciences, Peking University, China. 
\Letter~{\scriptsize\url{zwhao92@163.com}}}
}
\date{\today}

\begin{document}
	\maketitle
	\begin{abstract} 
We obtain several inequalities on the generalized means of dependent p-values. In particular, the weighted harmonic mean of p-values is strictly sub-uniform under several dependence assumptions of p-values, including independence, negative upper orthant dependence, the class of extremal mixture copulas, and some Clayton copulas. Sub-uniformity of the harmonic mean of p-values  has an important implication in multiple hypothesis testing: It is statistically invalid (anti-conservative) to merge p-values using the harmonic mean unless a proper threshold or multiplier adjustment is used,
and this applies across all significance levels.
The required multiplier adjustment on the harmonic mean {p-value grows sub-linearly to infinity} as the number of p-values increases, and hence there does not exist a constant multiplier that works for any number of p-values, even under independence.

	\end{abstract}


\section{Preamble: Connection between merging p-values and risk aggregation}\label{sec:0}

 This paper studies  merging p-values  via averaging in the context of testing a global hypothesis.
 We first explain the connection  between two areas: merging p-values and risk aggregation, before a formal introduction of the research problems and our contributions. 

The combined p-values of many merging methods take the following form $$ M_{g}(U_{1},\dots,U_{n})=g^{-1}\(\frac{1}{n}\sum_{i=1}^{n}g(U_{i})\),$$ where the random variables $U_1,\dots,U_n$ represent possibly dependent p-values and $g:[0,1]\rightarrow[-\infty,\infty]$ is a strictly monotone and continuous function. For instance, the harmonic mean merging method \citep{W19} corresponds to choosing $g$ as the mapping $p\mapsto 1/p$ on $(0,1)$. 
 As $M_{g}(U_{1},\dots,U_{n})$ may not be a valid p-value, i.e., $\p(M_{g}(U_{1},\dots,U_{n})\le p)\le p$ may not hold for some significance level $p\in(0,1)$, it is imperative to determine a valid threshold, denoted by $t_p$, such that $\p(M_{g}(U_{1},\dots,U_{n})\le t_p)\le p$.
 
 If $g$ is   increasing, we can rewrite $\p(M_{g}(U_{1},\dots,U_{n})\le t_p)$ as  $\p(\sum_{i=1}^{n}g(U_{i})\le n  g( t_p))$; the case with   decreasing $g$ is similar. Hence, the problem of deciding 
statistically valid thresholds $t_p$ for the combined p-value 
can be converted into the computation of
\begin{equation} \label{eq:QRM}
\mbox{the quantile  function of the test statistic } g(U_1)+\dots+g(U_n) \mbox{ or   bounds on it}.
\end{equation}   
Although bearing very different motivations, problem \eqref{eq:QRM} has a prominent presence in   Quantitative Risk Management (QRM), known as the problem of risk aggregation, which  concerns   quantitative assessments, such as a quantile or a risk measure, of the sum of random variables  representing financial losses. We refer to \citet[Chapter 8]{MFE15} for an introduction to risk aggregation. 
This connection allows us to utilize many existing results and techniques developed in QRM to study the problem of merging p-values. 
There are three connected streams of research on risk aggregation in QRM that are particularly relevant to merging p-values. 
\begin{enumerate}[(a)]
    \item The first one is Extreme Value Theory (EVT) in risk aggregation, because each $g(U_i)$ in \eqref{eq:QRM} typically has a heavy-tailed, often even infinite-mean, distribution (as in the case of the harmonic mean merging method).  For instance, 
\cite{ELW09} obtained results on the asymptotic behaviour of the quantile of the aggregate risk via EVT,  thus computing \eqref{eq:QRM} in an asymptotic sense, and this helps us to understand the limiting behaviour of the combined p-value as the significance level goes to $0$. We refer to \cite{EKM97} for a comprehensive treatment of EVT in finance and insurance.
The case of our main interest, the harmonic mean of uniform random variables, is discussed in  \citet[Example 7]{embrechts2002correlation} for $n=2$. 
\item The second one is dependence modeling for risk aggregation with copulas. 
Copula models are very popular in the statistical and probabilistic modeling of aggregate risks, and they help to analyze merging p-values in \eqref{eq:QRM} with certain classes of dependence structures. 
We refer to \cite{FV98},  \cite{DGM99}, \cite{CEN06}, and \cite{E09} for using copulas in risk aggregation in various contexts of risk management. The textbook of \cite{MFE15} provides a general treatment. 
\item 
The third one is quantile bounds on dependence uncertainty  in risk aggregation. In the context of merging p-values, this corresponds to the problem of averaging arbitrarily dependent p-values.
For instance,   results in 
\cite{EP06}, \cite{WPY13}, \cite{EWW15}, and \cite{BMWW16} on dependence uncertainty, as well as the algorithm of \cite{EPR13}, are essential to the generalized mean of p-values under arbitrary dependence among p-values. We refer to \cite{EPRWB14} for a summary.
\end{enumerate} 
The above selected references are subjective and may be slightly biased towards the work of Paul Embrechts, who has been a  key figure in all streams of research above,  not only with tremendous research contributions but also with remarkable mentorship for many younger scholars in the field. 

This paper continues to enhance the connection between the two fields by studying merging p-values, especially the harmonic mean of p-values, through recent developments in risk modeling. Some of our results are either directly built on or inspired by the work of Embrechts.

  \section{Introduction}\label{sec:1}

In multiple testing of a single hypothesis and testing multiple hypotheses,
a decision maker often needs to combine several p-values into one p-value. 
 {Recently, \cite{W19} proposed a statistical procedure based on the harmonic mean of p-values, which belongs to the larger class of
 generalized mean p-values studied by \cite{VW20}.} The class of generalized mean p-values also includes Fisher's combination method \citep{F48} via the geometric mean, often applied under the assumption that p-values are independent. 
 The harmonic mean p-value  method of \cite{W19}  has some desirable properties such as being applicable under a wide range of dependence assumptions of p-values, and has received considerable attention in statistics and the natural sciences. 
Validity, admissibility, and threshold adjustments of the generalized mean methods for p-values with arbitrary dependence are studied further by   \cite{VWW21} and \cite{CLTW21}.

We say that a combined p-value, denoted by $P$, is conservative/valid (resp.~anti-conservative/invalid) if $\p(P\le p)\le p$ for all $p\in(0,1)$ (resp.~$\p(P\le p)> p$ for some $p\in (0,1)$); in practice one may only be interested in a specific value of $p$ such as $0.01$ or $0.05$.
Here, $\p$ is a fixed probability under the null hypothesis of interest. 
The  {harmonic mean of p-values, also called the harmonic mean p-value,\footnote{The harmonic mean p-value method usually refers to the statistical procedure proposed by \cite{W19}.}  is known to be anti-conservative under some dependence structures, as noted by \cite{W19}. 
 If the underlying dependence structure is arbitrary, a threshold correction of order $\log n$ is needed, where $n$ is the number of p-values to merge \citep{VW20}. 
This correction generally leads to very conservative tests (i.e., the Type-I error rates are very small), and it may be reduced or even omitted under some specific dependence assumptions.
   In this paper, we study a stochastic order relation between a weighted generalized mean of standard uniform p-values and a standard uniform p-value under several dependence assumptions, and discuss its implications for the validity and threshold adjustment for harmonic mean p-values.

   Let $
\Delta_n =\{(w_1,\dots,w_n)\in [0,1]^n: w_1+\dots+w_n=1\} 
   $ be the unit $n$-simplex.
   We always assume $n\ge 2$.
  For $r\in \R\setminus \{0\}$, $n\in \N$, and $\mathbf w=(w_1,\dots,w_n)\in \Delta_n$,  the (weighted) $r$-mean  function is defined as
$$
M^{\mathbf w}_r(u_1,\dots,u_n) = \left(w_1 u_1^r+\dots+w_n u_n^r \right)^{1/r},~~~~~(u_1,\dots,u_n) \in (0,\infty)^{n}.
$$ 
The $0$-mean function is the weighted geometric mean, that is,
$
M^{\mathbf w}_0(u_1,\dots,u_n) =\prod_{i=1}^n u_i^{w_i},
$
 which is also the limit of $M^{\mathbf w}_r$ as $r\to 0$. The $(-1)$-mean function is referred to as the \emph{harmonic mean}.
 If $w_1=\dots=w_n=1/n$, $M_{r}^{\mathbf w}$ is the symmetric $r$-mean function, denoted by $M_r$ for simplicity, and it  is defined on $\bigcup_{n\in \N}(0,\infty)^n$. Denote by $\Delta^+_n=\Delta_n\cap (0,1)^n$.  For some of our results, we will only consider $\mathbf w\in \Delta^+_n$ since if some components of $\mathbf w$ are zero, we can simply reduce the dimension of $M_{r}^{\mathbf w}$.

Throughout, $U_1,\dots,U_n$ are (standard) uniform random variables on $(0,1)$ that are possibly dependent, and they represent p-values to combine.  The quantity $M^{\mathbf w}_{r} (U_1,\dots,U_n)$ is the weighted $r$-mean of p-values.
For two random variables $X$ and $Y$, we say $X$ is less than $Y$ in stochastic order, denoted by $X\preceq_{\mathrm st} Y$, if $\p(X\le x)\ge \p(Y\le x)$ for all $x\in \R$. Moreover, we write $X\simeq_{\mathrm st} Y$ if $X$ and $Y$ have the same distribution.
 The main results in this paper concern the following inequality under several dependence assumptions of $U_1,\dots,U_n$
\begin{align}\label{eq:maineq}
M^{\mathbf w}_{r} (U_1,\dots,U_n) \preceq_{\mathrm st} U_1,
\end{align}
where $r \le- 1$. Relation  \eqref{eq:maineq} is quite strong as it requires   $\p(M^{\mathbf w}_{r} (U_1,\dots,U_n)\le p)\ge p$ to hold for all threshold levels $p\in(0,1)$. Note that \eqref{eq:maineq} cannot hold for $r >-1 $ except for identical $U_1,\dots,U_n$ (see Proposition \ref{prop:trivial}).

 A non-negative random variable $X$ is said to be  \emph{sub-uniform} if $X\preceq_{\mathrm st} U_1$.  
Moreover, $X$ is \emph{strictly sub-uniform} if 
\begin{align}\label{eq:subunif-3}
\p(X\le p ) >p ~\mbox{ for all $p\in (0,1)$}.
\end{align} 
Using a sub-uniform p-value is anti-conservative in hypothesis testing, since it has a larger type-I error rate than the nominal level. Therefore, if 
\eqref{eq:maineq} holds true, then   merging p-values using the harmonic mean, or any $r$-generalized mean function with $r\le -1$, is anti-conservative across all significance levels  in $(0,1) $.  
\begin{remark}[Terminology]
Although sub-uniformity is  an important property for studying p-values, 
this term   has been used  with different meanings  in the literature. Some of them are collected here. A non-negative random variable $X$ is called  super-uniform by \cite{barber2017p} if $U_1\preceq_{\mathrm st} X$ (anticipating that sub-uniformity should be defined by flipping the above inequality, their terminology is consistent with ours), but such a random variable is called sub-uniform by \cite{ferreira2006benjamini}.  Moreover, \cite{chen2020benjamini} defined sub-uniformity  in the strict sense  \eqref{eq:subunif-3}.  \cite{rubin2019meta} called $X$ sub-uniform if it is dominated by $U_1$ in  convex order.\footnote{A random variable $X$ is said to be dominated by $Y$ in convex order if $\E[f(X)]\le\E[f(Y)]$ holds for all convex functions $f$, provided that the expectations exist.} 
\end{remark}

When mentioning (strict) sub-uniformity later in this paper, we always refer to the corresponding property of $M^{\mathbf w}_{r}(U_1,\dots,U_n)$ or $M_r(U_1,\dots,U_n)$ with $r\le -1$, which will be clear from the context.

The main objective of this paper is to study  \eqref{eq:maineq}  given several dependence assumptions of $U_1,\dots,U_n$, including negative upper orthant dependence (Proposition \ref{th:NLD}), the class of extremal mixture copulas (Theorem \ref{prop:extremal}), and some Clayton copulas (Theorem \ref{th:Clayton}). 
Some of our results are built on a recent study of \cite{CS24}, where a stochastic ordering inequality on heavy-tailed random variables is established. 
As discrete p-values may arise in hypothesis testing (e.g., \cite{VGS05}), we also study  sub-uniformity for discrete uniform random variables on $\{1/m,\dots,m/m\}$ for $m\in\N$. This situation is quite different from the uniform case as \eqref{eq:maineq} can never hold for $\mathbf w\in\Delta_n^+$ in the case of  discrete uniform random variables unless they are identical. However, using the harmonic mean function to merge discrete uniform random variables can still be anti-conservative  at some threshold levels (Theorem \ref{th:discrete}).

Most findings of this paper are 
 negative results: Under many different assumptions of dependence among p-values, the harmonic mean p-value is anti-conservative and cannot be used without a proper adjustment, and the adjustment coefficient diverges even in the case of independence. 
 To address these issues, while keeping the advantages of the harmonic mean p-value, it is recommendable to use the Simes method \citep{S86} or the Cauchy combination method \citep{LX20}, which are shown to be valid under various forms of dependence, and perform comparably to the harmonic mean p-value; see results in \cite{CLTW21} on comparing these three methods. Other methods based on heavy-tailed transformation of p-values can also be used under different assumptions \citep{GJW23}.
 An exception to the above negative results is the case of Clayton copulas treated in Theorem \ref{th:Clayton}, where we obtain a positive result that the harmonic mean p-value can be made valid with a small threshold adjustment (a multiplicative factor of $1.131$) under the assumption of Clayton copulas with parameter at least $1$.

The rest of the paper is organized as follows. In Section \ref{sec:2}, we first discuss the intuition behind \eqref{eq:maineq} in the simplest case, and  then present general properties related to sub-uniformity for dependent $U_1,\dots,U_n$. In Section \ref{sec:3},   \eqref{eq:maineq} is shown given the aforementioned dependence assumptions of $U_1,\dots,U_n$.
In Section \ref{sec:5}, the threshold of the harmonic mean p-value is studied for independent p-values, where we see that the adjustment increases at the rate of $\log n$ for a fixed probability level. Sub-uniformity for discrete uniform random variables is studied in  Section \ref{sec:4}. Numerical examples based on simple simulations are presented in Section \ref{sec:6}, and Section \ref{sec:7} concludes the paper.

 \section{Some observations and general results on sub-uniformity}\label{sec:2}

We set straight some first observations on the sub-uniformity inequality \eqref{eq:maineq}.
First, the generalized mean is monotone in $r$; that is, given any  $\mathbf w\in \Delta_n$, $M^{\mathbf w}_r\le M^{\mathbf w}_s$ for $r\le s  $ (Theorem 16 of \cite{HLP34}). Hence,  we have $M^{\mathbf w}_{ r} (U_1,\dots,U_n)  \le M^{\mathbf w}_{-1} (U_1,\dots,U_n) $ for all $r \le - 1$.
To get the  sub-uniformity inequality \eqref{eq:maineq} for all $r\le -1$, it suffices to show  
  \begin{equation}\label{eq:1} 
\p(M^{\mathbf w}_{-1} (U_1,\dots,U_n) \le p) \ge p \mbox{~for all $p\in (0,1)$}.
\end{equation} 
This observation simplifies our journey by allowing us to focus on the case of the harmonic mean, which also happens to be the most popular method within the class of generalized mean methods for $r\in(-\infty,0)$.

We begin with a simple proof for independent $U_1,\dots,U_n$ in the symmetric case. Although \eqref{eq:1} in this case directly follows from Theorem 1 of \cite{CEW24} or Theorem 1 of \cite{CS24}, the proof below, different from \cite{CEW24} and \cite{CS24}, helps to understand  \eqref{eq:1} via a result well-known in the multiple testing literature, namely the Simes inequality \citep{S86}. For $n\in\N$, write $[n]=\{1,\dots,n\}$. For $(u_1,\dots,u_n) \in (0,\infty)^{n}$, the Simes function is defined as $S(u_1,\dots,u_n)= \min_{i\in[n]}\{ n u_{(i)}/i\}$ where $u_{(1)},\dots,u_{(n)}$ are the order statistics of $u_1,\dots,u_n$, from the smallest to the largest. As shown by \cite{S86},
$S(U_1,\dots,U_n)$ is uniformly distributed on $(0,1)$ given independent
 $U_1,\dots,U_n$. Moreover, we have $M_{-1}\le S$; this inequality is in  Theorem 3 of \cite{CLTW21} and was also discussed in Section 4 of the online appendix of \cite{W19}, but one can also check it directly.
Putting these two observations together, we have   
\begin{equation*}
M_{-1}(U_1,\dots,U_n)\le S(U_1,\dots,U_n)\simeq_{\mathrm st} U_1,
\end{equation*} 
which implies that \eqref{eq:maineq} holds for symmetric mean of independent $U_1,\dots,U_n$. 

 Before showing \eqref{eq:maineq} holds under more general assumptions, we discuss several properties related to the target problem. 
We first explain that it is only meaningful to consider the sub-uniformity \eqref{eq:maineq}  for $r\le - 1$.
Indeed, as illustrated in Proposition \ref{prop:trivial} below, 
 sub-uniformity can hold for some $r>-1$ only in the trivial case that $U_1,\dots,U_n$ are identical, and thus strict sub-uniformity can never hold  for any $r>-1$. 
\begin{proposition}\label{prop:trivial}
The following statements are equivalent. 
\begin{enumerate}[(i)]
\item $M^{\mathbf w}_{r} (U_1,\dots,U_n) \preceq_{\mathrm st} U_1 $ for some  $r\in (-1,\infty)$ and   $\mathbf w\in \Delta_n^+$;
\item $M^{\mathbf w}_{r} (U_1,\dots,U_n) \simeq_{\mathrm st} U_1 $ for some   $r\in (-1,\infty)$ and   $\mathbf w\in \Delta_n^+$;
\item $U_1=\dots=U_n$ a.s.;
\item $ 
M^{\mathbf w}_{ r} (U_1,\dots,U_n) = U_1 $ a.s.~for all $r \in \R$ and all  $\mathbf w\in \Delta_n$;
\item $ 
M^{\mathbf w}_{r} (U_1,\dots,U_n) = U_1 $ a.s.~for some $r\in (-1,\infty)$ and   $\mathbf w\in \Delta_n^+$.
\end{enumerate} 
\end{proposition}
\begin{proof}
Note that the binary relation $X \preceq_{\mathrm st} Y$ is flipped under decreasing transformation on both $X$ and $Y$. Hence, for $r <0$, we can write $M^{\mathbf w}_{r} (U_1,\dots,U_n) \preceq_{\mathrm st} U_1 $ as 
\begin{align*}
  \sum_{i=1}^n w_iU_i^{r}  \succeq_{\mathrm st}    U_1^{r}.
\end{align*}  The case $r\ge  0$ can be argued similarly, and we will omit it from the discussion below.

 \underline{(i)$\Rightarrow$(ii)}:
If $ \sum_{i=1}^n w_i U_i^{r} $  and  $U_1^{r}$ are not identically distributed, then 
$$
 \E\left[  \sum_{i=1}^n w_i U_i^{r} \right]  > \E\left[   U_1^{r}\right],
$$
which leads to a contradiction, since both expectations are equal to $1/(r+1)$. 
Hence, we have $  \sum_{i=1}^n w_i U_i^{r} \simeq_{\mathrm st} U_1^{r}$.
   
  \underline{(ii)$\Rightarrow$(iii)}: 
Let $U=M_{r}^{\mathbf w}(U_1,\dots,U_n)$, which is uniformly distributed on $[0,1]$.
Note that 
\begin{align}\label{eq:pr1-pf0}
  \sum_{i=1}^n w_i U_i^{r}  =   U^{r}.
\end{align}
Take $\epsilon \in (0,1+r)$. From \eqref{eq:pr1-pf0}, we have 
\begin{align}\label{eq:pr1-pf1}  \E\left[  \sum_{i=1}^n w_iU_i^{r} U^{-\epsilon} \right]   =  \E\left[   U^{r  -\epsilon }\right] = \frac{1}{1+r-\epsilon}.\end{align}
For all $i\in[n]$,
 by Lemma 7.27 of \cite{MFE15} and the Fr\'echet-Hoeffding inequality, we have  
 $ \E [ U_i^{r} U^{-\epsilon} ]\le  \E[   U^{r-\epsilon}]$, and the equality holds if and only if $U_i^{r}$ and $U^{-\epsilon}$ are comonotonic.
Together with \eqref{eq:pr1-pf1}, we obtain $ \E [ U_i^{r} U^{-\epsilon} ] =  \E[ U^{r-\epsilon}]$ and thus $U_i^{r}$ and $U^{-\epsilon}$ are comonotonic.
 Moreover, as $U_i$ and $U$ are identically distributed and $r<0$, $U_i=U$ a.s.~for all $i\in[n]$.

The remaining implications, \underline{(iii)$\Rightarrow$(iv) $\Rightarrow$(v) $\Rightarrow$(i)},
are straightforward by definition.  
\end{proof}

 Joint distributions of standard uniform random variables are known as copulas; see  \cite{N06} for an introduction to copulas. Next, we explain how to construct copulas for which \eqref{eq:maineq} holds. We write $X\overset{\mathrm{d}}{\sim}F$ if a random variable or random vector $X$ is distributed as $F$. Fix $r\le -1$.
In what follows, 
we say sub-uniformity holds for a copula $C$ (or a random vector $(U_1,\dots,U_n)$), if $M^{\mathbf w}_{r} (U_1,\dots,U_n) \preceq_{\mathrm st} U_1$ holds for $(U_1,\dots,U_n)\overset{\mathrm{d}}{\sim} C$ and  all $\mathbf w\in\Delta_n$; 
by saying that  strict sub-uniformity holds for a copula $C$ (or a random vector $(U_1,\dots,U_n)$), we mean that 
 \begin{equation*}
\p(M^{\mathbf w}_{r} (U_1,\dots,U_n) \le p) > p \mbox{~for all $p\in (0,1)$},
\end{equation*}
 holds for $(U_1,\dots,U_n)\overset{\mathrm{d}}{\sim} C$ and  all $\mathbf w\in\Delta_n^+$. 
Let $\mathcal C_k$ be the set of all $k$-dimensional copulas, $k\in \N$.

\begin{proposition}\label{prop:trivial2}
Let $r\le -1$ and $\mathcal C \subseteq \mathcal C_n$. If  sub-uniformity holds for   each copula in $\mathcal C$, then it holds for any copula in the convex hull of $\mathcal C$. 
\end{proposition}
\begin{proof}
Note that for all $p\in(0,1)$ and any $\mathbf w\in\Delta_n$, $\p(M^{\mathbf w}_{r} (U_1,\dots,U_n) \le p)$ is linear in the distribution of $(U_1,\dots,U_n)$. Since  sub-uniformity holds for every element in $\mathcal C$, it also holds for every element from the convex hull of $\mathcal C$.
\end{proof}

The following proposition shows that sub-uniformity can be passed from  smaller groups to a larger group 
of p-values in two different ways.
In what follows, 
for vectors $\mathbf x=(x_1,\dots,x_n)\in\R^n$ and $\mathbf y=(y_1,\dots,y_n)\in\R^n$, their dot product is $\mathbf x\cdot \mathbf y=\sum_{i=1}^nx_iy_i$, and we denote by $\|\mathbf x\|=\sum_{i=1}^n|x_i|$ and $\mathbf x^{-1}=(x_1^{-1},\cdots,x_n^{-1})$.
Moreover, $\bigwedge \mathbf x =\min\{x_1,\dots,x_n\}$.
\begin{proposition}\label{prop:trivial3}
Let $r\le -1$, and $C_i \in \mathcal C_{k_i}$ with   $k_i\in \N$ and $i\in[n]$. 
\begin{enumerate}[(i)]
\item 
 If  sub-uniformity holds for $C_i$ for all $i\in[n]$, then it holds for 
 $C(\mathbf u_1,\dots,\mathbf u_n)=\prod_{i=1}^n C_i(\mathbf u_i)$, $\mathbf u_i\in \R^{k_i}$, $i\in[n]$.
 \item
 Suppose that $C_i(\mathbf u_i)=\bigwedge  \mathbf u_i$, $i\in[n]$. If  sub-uniformity holds for $C^{*}\in\mathcal C_{n}$, then it holds for $C(\mathbf u_1,\dots,\mathbf u_n)=C^{*}(C_1(\mathbf u_1),\dots, C_n(\mathbf u_n))$,  $\mathbf u_i\in \R^{k_i}$, $i\in[n]$.
\end{enumerate}
\end{proposition}
\begin{proof}
Suppose that $(\mathbf U_1,\dots,\mathbf U_n)\overset{\mathrm{d}}{\sim} C$ such that $\mathbf U_i\overset{\mathrm{d}}{\sim} C_i$, $i\in[n]$.
\begin{enumerate}[(i)]
\item  Let $\mathbf w_i\in [0,1]^{k_i}$, $i\in[n]$, such that $(\mathbf w_1,\dots,\mathbf w_n)\in \Delta_{\sum_{i=1}^nk_i}$. Let $U,V_1,\dots, V_n$ be iid standard uniform random variables. Since  sub-uniformity holds for $C_i$,  we have $\mathbf w_i\cdot\mathbf U_i^{-1} \succeq_{\mathrm st} \|\mathbf w_i\| V_i^{-1}$ for all $i\in[n]$. As $C(\mathbf u_1,\dots,\mathbf u_n)=\prod_{i=1}^n C_i(\mathbf u_i)$, $\mathbf U_1,\dots, \mathbf U_n$ are independent of each other. As stochastic order is preserved under convolution (e.g., Theorem 1.A.3 in \cite{SS07}), $\sum_{i=1}^n\mathbf w_i\cdot\mathbf U_i^{-1} \succeq_{\mathrm st}\sum_{i=1}^n \|\mathbf w_i\| V_i^{-1}$. Moreover, by Theorem 1 in \cite{CEW24}, $\sum_{i=1}^n \|\mathbf w_i\| V_i^{-1}\succeq_{\mathrm st}U^{-1}$. Consequently, we have the desired result as follows
$$\sum_{i=1}^n\mathbf w_i\cdot\mathbf U_i^{-1} \succeq_{\mathrm st}\sum_{i=1}^n \|\mathbf w_i\| V_i^{-1}\succeq_{\mathrm st}U^{-1}.$$
\item Since $C_i(\mathbf u_i)=\bigwedge \mathbf u_i$, the components of $\mathbf U_i$, $i\in[n]$, are perfectly positively dependent (i.e., they are almost surely equal as they follow the same distribution). The desired result is obvious as  sub-uniformity holds for $C^{*}$.\qedhere
\end{enumerate}
\end{proof}

By Proposition \ref{prop:trivial3} (i),  if sub-uniformity holds for independent subgroups of standard uniform random variables, it also holds for the whole group. Proposition \ref{prop:trivial3} (ii) says that,  for a group of standard uniform random variables that consists of $n$ subgroups of perfectly positively dependent components, if  sub-uniformity holds for $n$ components each of which comes from one distinct subgroup, then it also holds for the whole group.

In the rest of this paper, we will show that sub-uniformity holds for several dependence models including independence and some forms of positive and negative dependence.
Together with Propositions \ref{prop:trivial2} and 
\ref{prop:trivial3}, by mixing these dependence models or connecting them via groups, we can obtain a much wider range of dependence models for which sub-uniformity holds.

%
  

   \section{Sub-uniformity for dependent p-values}\label{sec:3}
 
  In this section, we study  sub-uniformity for standard uniform random variables that are negatively or positively dependent in specific forms. 
  In particular, 
  we show that sub-uniformity holds for negative upper orthant dependence, extremal mixture copulas, and some Clayton copulas. 
  
   \subsection{Negative upper orthant dependence}
   A random vector $\mathbf X= (X_1,\dots,X_n)$  is \emph{negatively upper orthant dependent} (NUOD) if for all $x_1,\dots,x_n\in \R$, 
 it holds that 
\begin{equation}\label{eq:NUOD}
\p(X_1> x_1,\dots,X_n> x_n)\le \prod_{i=1}^n\p(X_i> x_i).
\end{equation}
 It is said to be \emph{negatively lower orthant dependent} (NLOD) if for all $x_1,\dots,x_n\in \R$, 
\begin{equation}\label{eq:NLOD}
 \p(X_1\le x_1,\dots,X_n\le x_n)\le \prod_{i=1}^n\p(X_i\le x_i).
 \end{equation}
 In general, the two notions of negative dependence are not equivalent except when $n=2$. If $\mathbf X$ is both NUOD and NLOD, it is said to be \emph{negatively orthant dependent}, as introduced by \cite{block1982some}. Negative orthant dependence includes multivariate normal distributions with non-positive correlations as special cases; see, e.g., \cite{chi2022multiple} for examples of negative dependence in multiple hypothesis testing.

 Intuitively, \eqref{eq:NUOD} (resp.~\eqref{eq:NLOD}) means that that $X_1,\dots,X_n$ are less likely to be large (resp.~small) simultaneously compared to $X_1',\dots,X_n'$, which are iid copies of $X_1,\dots,X_n$, hence is a notion of negative dependence. Negative orthant dependence is closely related to other notions of negative dependence. It is implied by 
negative association \citep{JP83},  negative regression dependence \citep{block1985concept}, and weak negative association \citep{CEW24}; see, e.g., \cite{CEW24} for a discussion of these implication relations.

The next result gives sub-uniformity for NUOD standard uniform random variables.  It is based on Theorem 1 of  \cite{CS24}, where a stochastic order relation for heavy-tailed random variables is shown. As such, this result is essentially known in a different context, but we present it here for its concise statement and clear interpretation.

 \begin{proposition}\label{th:NLD}
 Let $r \le-1$. If $(U_1,\dots,U_n)$ is NUOD, then $M_{r}^{\mathbf w}(U_1,\dots,U_n)$ is strictly sub-uniform for all  $\mathbf w\in \Delta^+_n$.
\end{proposition}
\begin{proof}
   It suffices to show $$
  \p\(\sum_{i=1}^n w_iU_i^{-1}>t\)>\frac{1}{t}\mbox{~~holds for all $t>1$ and all $\mathbf w\in\Delta_n^+$.}$$ 
  Note that $U_i^{-1}-1$, $i\in[n]$, follows a Pareto distribution with distribution function $\p(U_i^{-1}-1\le x)=1-1/(x+1)$, $x> 0$. As $(U_1,\dots,U_n)$ is NUOD, $(U_1^{-1},\dots,U_n^{-1})$ is NLOD. By Theorem 1 of \cite{CS24}, 
  $$ \p\(\sum_{i=1}^n w_i(U_i^{-1}-1)>x\)>\p\(U_1^{-1}-1>x\)=\frac{1}{x+1}\mbox{~~holds for all $x>0$ and all $\mathbf w\in\Delta_n^+$.}$$
  Hence, we have the desired result.
\end{proof}
Certainly, the result in Proposition \ref{th:NLD} also holds for 
 $(U_1,\dots,U_n)$ satisfying any of the stronger notions of negative dependence than negative upper orthant dependence. 

 \subsection{Extremal mixture copulas}
 Next, we apply Proposition \ref{th:NLD} to show that sub-uniformity holds for extremal mixture copula \citep{mcneil2020attainability}. We say that $(U_1,\dots,U_n)$ follows an \emph{extremal copula} $C\in\mathcal C_{n}$ with an index set $ J\subseteq [n]$, if
$$U_{j}\overset{\mathrm d}{=}\id_{\{j\in J\}}U+\id_{\{j\in J^{c}\}}(1-U),\mbox{~~for~~} j\in[n],$$ where $U$ is a standard uniform random variable. For $n\geq 2$, there are $2^{n-1}$ different extremal copulas. Let $\mathbf s_{i}$ be a vector consisting of the digits of a $n$-digit binary number which represents the decimal number $i-1$, for each $i\in [2^{n-1}]$. For instance, if $n=3$, we have
$$\mathbf s_{1}=(0,0,0),\mbox{~~}\mathbf s_{2}=(0,0,1),\mbox{~~}\mathbf s_{3}=(0,1,0),\mbox{~~}\mathbf s_{4}=(0,1,1).$$
 For $n\in\N$, let $J_i$ be the index set of zeros in $\mathbf s_{i}$, for each $i\in[2^{n-1}]$. Denote by $C^{(i)}$ the extremal copula with index set $J_i$. Note that $C^{(1)}$ is the comonotonicity copula.    A copula $C$ is an \emph{extremal mixture copula} with a vector $(a_1,\dots,a_{2^{n-1}}) \in\Delta_{2^{n-1}}$ if $C=\sum_{i=1}^{2^{n-1}}a_{i}C^{(i)}$. A random vector following an extremal mixture copula is not necessarily NUOD.

\begin{theorem}\label{prop:extremal}
Let $r \le-1$. If $(U_1,\dots,U_n)$ follows an extremal mixture copula with $(a_1,\dots,a_{2^{n-1}})\in \Delta_{2^{n-1}}$ such that $a_1<1$, then $M_{r}^{\mathbf w}(U_1,\dots,U_n)$ is strictly sub-uniform for all $\mathbf w\in \Delta^+_n$.
\end{theorem}
\begin{proof}
Let $(V_1,\dots,V_n)$ follow an extremal copula $C^{(k)}\in\mathcal C_{n}$ with an index set $ J_{k}\subseteq [n]$, $k\in[2^{n-1}]$. For $\mathbf w=(w_{1},\dots,w_{n})\in \Delta^+_n$ and $p\in(0,1)$, 
$$\p(M^{\mathbf w}_{-1} (V_1,\dots,V_n) \le p)=\p\(\sum_{i\in J_{k}}w_{i}V^{-1}+\sum_{j\in J_{k}^{c}}w_{j}(1-V)^{-1}\ge p^{-1}\)=\p(M^{\mathbf \eta}_{-1} (V,1-V) \le p),$$
where $\mathbf \eta=(\sum_{i\in J_{k}}w_{i},\sum_{j\in J_{k}^{c}}w_{j})\in \Delta_{2}$ and $V$ is a standard uniform random variable. Since $C^{(1)}$ is the comonotonicity copula, the above probability equals to $p$ if $k=1$. It is straightforward to check that $(V,1-V)$ is NUOD.
By Proposition \ref{th:NLD}, $M^{\mathbf \eta}_{-1} (V,1-V)$ is strictly sub-uniform if $k\ge 2$. Hence,  strict sub-uniformity holds for extremal copula $C^{(k)}$, $k=2,\dots,2^{n-1}$. As extremal mixture copula is a weighted mixture of extremal copulas, by Proposition \ref{prop:trivial2},   sub-uniformity holds for any extremal mixture copula. It is clear that  sub-uniformity is strict if $a_1\neq1$.
\end{proof}

\begin{remark}
 By \citet[Theorem 1]{mcneil2020attainability}, every Kendall's rank correlation matrix (Kendall matrix, for short) can be attained by some extremal mixture copula. Together with Theorem \ref{prop:extremal}, this implies that for any given Kendall matrix $\mathcal T$ and $r\le -1$, 
there always exists a copula with the specified Kendall matrix $\mathcal T$ such that 
 sub-uniformity holds, and strict sub-uniformity holds if $\mathcal T$ is not the matrix of ones.
\end{remark}
  
  \subsection{Clayton copula and positive dependence}
 {In this section, we study sub-uniformity for standard uniform random variables with a specific positive dependence structure, modelled by  Clayton copulas, which have been used to model p-values in, e.g., \cite{DG13}, \cite{BD14}, and \cite{neumann2019multivariate}.}  
  The Clayton copula $C$ with parameter $t>0$, denoted by $\mathrm{Clayton}(t)$, is   given by
$$
C(u_1,\dots,u_n) = \left(u_1^{-t}+\dots+u_n^{-t}- (n-1)\right)^{-1/t},~~ (u_1,\dots,u_n)\in(0,1)^n.
$$
The Clayton copula with $t>0$ represents a type of positive dependence, with $t\to \infty $ yielding comonotonicity and $t\downarrow0$ yielding independence.
 For a random vector following a Clayton copula, Kendall's tau of any pair of its components is equal to $t/(t+2)$. 
 
 Clayton copulas arise naturally in the following context.  
Suppose that $X_1,\dots,X_n$ are iid exponential random variables with parameter $\lambda$,  and $Y$ is a Gamma random variable with parameter $(1/t,1)$ independent of $X_1,\dots,X_n$; the exponential distribution with parameter $\lambda>0$ is given by $F(x)=1-\exp(-x/\lambda)$, $x\ge0$ and the Gamma distribution with parameter $(k,\theta)\in\R_+^2$ is given by $F(x)=(\Gamma(k)\theta^k)^{-1}\int_0^x y^{k-1}\exp(-y/\theta)\mathrm{d}y$, $x\ge 0$. 
The random vector $(T_1,\dots,T_n)=\({X_1}/{Y},\dots,{X_n}/{Y}\) $ is usually used to model the lifetimes of $n$ objects in a system; see, e.g., \cite{lindley1986multivariate}. 
   The joint distribution of  $(T_1,\dots,T_n)$ is known as a multivariate Pareto distribution of type II with marginal distribution 
   $G(x)=1-\({\lambda}/{(\lambda+x)}\)^{1/t}$ for $ x\ge 0.$ 
Let 
\begin{equation}\label{eq:rw-2} (U_1,\dots,U_n)=(1-G(T_1),\dots,1-G(T_n))
=(1-G(X_1/Y),\dots,1-G(X_n/Y)).
\end{equation} Each of $U_1,\dots,U_n$ follows a standard uniform distribution and $(U_1,\dots,U_n)$   has a Clayton($t$) copula; see, e.g., \cite{sarabia2016risk}.

The next result gives sub-uniformity of the   $r$-mean of p-values following a corresponding Clayton copula, as well as a positive result on a type-I error rate bound for the $r$-mean of p-values.

\begin{theorem}\label{th:Clayton}
  Let $ t\ge1$, $\mathbf w\in \Delta^+_n$ and $(U_1,\dots,U_n)\overset{\mathrm{d}}{\sim}\mathrm{Clayton}(t)$.  If $r\ge t$, then $M_{-r}^\mathbf w(U_1,\dots,U_n)$ is strictly sub-uniform.
Moreover, for $p\in(0,1)$ and $s\in [1,t]$, we have
$$\p\(M^{\mathbf w}_{-s}(U_1,\dots,U_n)\le p\) \le G_{1/t} \left(\frac{1}{p^{-t}-1}\right),$$
where  $G_{1/t}$ is the cdf of a Gamma distribution with parameter $(1/t,1)$. In particular, 
\begin{align}\label{eq:rw-pos} \p\(M^{\mathbf w}_{-1}(U_1,\dots,U_n)\le p\) \le \sup_{b\ge 1}G_{1/b} \left(\frac{1}{p^{-b}-1}\right).\end{align}
\end{theorem}
\begin{proof} 




We first show the case when $t=r$.
Note that  for  $\mathbf w\in\Delta_n^+$, simple algebra leads to \begin{align}\label{eq:rw-1}
\p(M^{\mathbf w}_{-t}(U_1,\dots,U_n)\le p) = \p\left(\frac{\sum_{i=1}^n w_iX_i}{Y}\ge \lambda p^{-t}-\lambda \right),
\end{align}
where $X_1,\dots,X_n$ and $Y$ are as in \eqref{eq:rw-2}. 
Therefore, to show  sub-uniformity, i.e., 
$\p(M^{\mathbf w}_{-t}(U_1,\dots,U_n)\le p)>p$, 
it suffices to show $\p(\sum_{i=1}^n w_iX_i/Y\ge x) >\p( X_1/Y\ge x)$ for all $x>0$.
We have 
$$\p\(\frac{\sum_{i=1}^n w_iX_i}{Y}\ge x\)=\E\[\p\(Y\le \frac{\sum_{i=1}^n w_iX_i}{x}\mid \sum_{i=1}^n w_iX_i\)\]=\E\[\phi\(\sum_{i=1}^n w_iX_i\)\],$$
where $\phi(y)=\p(Y\le y/x)$ for $y>0$. Assume $x=1$ without losing generality. Taking the second derivative of $\phi$, we get 
$$\phi^{''}(y)=\frac{y^{1/t-2}\exp\(-y\)(1/t-1-y)}{\Gamma(1/t) }.$$
As $t=r\ge 1$, we have $\phi^{''}<0$ and $\phi$ is strictly concave. Therefore,
$$\p\(\frac{\sum_{i=1}^n w_iX_i}{Y}\ge x\)=\E\[\phi\(\sum_{i=1}^n w_iX_i\)\]> \E\[\phi\(X_1\)\]=\p\(\frac{X_1}{Y}\ge x\).$$
Hence, we obtain the desired stochastic dominance for $t=r$. The statement for $t\in[1,r)$ is due to the fact that  $M_{-r}\le M_{-t}$ \cite[Theorem 16]{HLP34}.

To show the last inequality, note that Jensen's inequality gives
$$\E\[\phi\(\sum_{i=1}^n w_iX_i\)\]
\le \phi \(\sum_{i=1}^n w_i\E[X_i]\) =\phi(\lambda) =\p(Y\le \lambda /x)=G_{1/t}(\lambda/x). $$
Hence, using \eqref{eq:rw-1},  
we get 
$$
\p(M_{-t}(U_1,\dots,U_n)\le p) 
\le G_{1/t}(1/(p^{-t}-1)),$$
and the desired upper bound follows from noting again that $M_{-s}\ge M_{-t}$ for $s\le t$.
\end{proof}

 Applying Theorem \ref{th:Clayton} with the special case $r=t=s=1$, we get that, if $(U_1,\dots,U_n)\overset{\mathrm{d}}{\sim}\mathrm{Clayton}(1)$,  then   
\begin{align*} 
p<\p\(M_{-1}(U_1,\dots,U_n)\le p\) \le 1-e^{1-1/(1-p)}
\le \frac{1}{1-p} -1  = \frac p {1-p}.
\end{align*}
As a consequence, although being anti-conservative without correction,    the harmonic mean p-value becomes valid with the simple threshold $t_p=p/(1+p)$, i.e., $\p\(M_{-1}(U_1,\dots,U_n)\le t_p\)\le p$ holds for any $n\in\N$.
The needed correction is minor since $p$ and $p/(1+p)$ are very close for small $p$.

The null p-values following
Clayton(1) is a strong assumption, and it can be relaxed by using \eqref{eq:rw-pos}.
Suppose that $(U_1,\dots,U_n)\overset{\mathrm{d}}{\sim}\mathrm{Clayton}(t)$ with some unknown $t\ge 1$. Define a constant 
 $$
\kappa =\sup_{p\in (0,0.1],~b\ge 1}\frac 1p G_{1/b} \left(\frac{1}{p^{-b}-1}\right).
 $$
By \eqref{eq:rw-pos}, we have
$\p\(M^{\mathbf w}_{-1}(U_1,\dots,U_n)\le p\) \le   \kappa p$   for all $p\in(0,0.1]$. 
  Numerical calculation gives $\kappa \approx 1.1304$; the maximum in computing $\kappa$ is approximately attained at $p=0.1$ and $b=2.0853$. Therefore, if $(U_1,\dots,U_n)\overset{\mathrm{d}}{\sim}\mathrm{Clayton}(t)$ with some $t\ge 1$, we can use a threshold $u_p=p/1.131$ for the harmonic mean p-value such that $\p(M^{\mathbf w}_{-1}(U_1,\dots,U_n)\le u_p)\le p$ for all $p\in(0,0.1]$. This correction is valid for all $n$. This shows that positive dependence makes the harmonic mean p-value well behaved; in sharp contrast, the needed correction {grows to infinity} as $n\to \infty$ in case of independence; see Section \ref{sec:5}. 
 
  The reason why the harmonic mean p-value behaves well for the Clayton copula is perhaps because $U_i^{-1}$, $i\in[n]$, is extremely heavy-tailed (i.e., $U_i^{-1}$ does not have a finite mean). It is well known in EVT that for iid extremely heavy-tailed random variables $U_1^{-1},\dots,U_n^{-1}$, a large value of $U_1^{-1}+\dots+U_n^{-1}$ (thus a small value of $M_{-1}(U_1,\dots,U_n)$)  is 
  most likely caused by 
 one extremely large  $U_i^{-1}$ rather than by several moderately large  $U_i^{-1}$ (see \cite{EKM97}). In the case of strong positive dependence, the possibility of getting one extremely large  $U_i^{-1}$ may be reduced (e.g., \cite{alink2004diversification}).

 In the case of the symmetric mean function  $M_r$, the distribution of $M_{-r}(U_1,\dots,U_n)$ with $(U_1,\dots,U_n)\overset{\mathrm{d}}{\sim}\mathrm{Clayton}(r)$ has an analytical formula provided below.
  
\begin{proposition}\label{lem:pareto-clayton}
Let $r\ge 1$ and $(U_1,\dots,U_n)\overset{\mathrm{d}}{\sim}\mathrm{Clayton}(r)$. For $p\in(0,1)$, we have
$$\p\(M_{-r}(U_1,\dots,U_n)\le p\)=1-B_{n,r}\(\frac{np^{-r}-n}{np^{-r}-n+1}\),$$
where $B_{n,r}$ is a Beta cdf given by 
\begin{equation*}
B_{n,r}(x)=\frac{\Gamma\(n+1/r\)}{\Gamma(n)\Gamma\(1/r\)}\int_{0}^{x}t^{n-1}(1-t)^{1/r-1}\mathrm{d}t,\mbox{~~}x\in(0,1).\end{equation*} 
\end{proposition}
\begin{proof}
By \eqref{eq:rw-1},  
we  have
\begin{align*}
\p\(M_{-r}(U_1,\dots,U_n)\le p\)=\p\(\sum_{i=1}^nT_i\ge \lambda (np^{-r}-n)\).
\end{align*}
Since $\lambda$ is the scale parameter of the exponential distribution, the above probability is indifferent to $\lambda$. We assume $\lambda =1$ for simplicity. Given $Y=y$, the conditional distribution of $\sum_{i=1}^nX_i/Y$ is a Gamma distribution with parameter $(n,1/y)$. Since $Y$ is also a Gamma distribution, $\sum_{i=1}^nT_i$ follows a compound Gamma distribution. Using (1.2) of \cite{dubey1970compound}, we have the desired equality. 
\end{proof}


Besides the above Clayton copulas, we provide below two other positive dependence structures for which
  sub-uniformity holds.
\begin{example}
Let $X,X_i\overset{\mathrm{d}}{\sim} \mathrm U(0,1)$, $Y,Y_i\overset{\mathrm{d}}{\sim} \mathrm U(\beta,1)$, and $Z,Z_i\overset{\mathrm{d}}{\sim} \mathrm U(0,\beta)$, $i\in[n]$, be independent, where $\beta \in(0,1)$.   Assume that $(U_1,\dots,U_n)$ is modelled by one of the two cases below, that is, 
\begin{align}\label{ex:1}
&U_i=\id_{\{X\le \beta\}}Z_i+\id_{\{X> \beta\}}Y, ~~i\in[n]  
~~~~\mbox{or}~~~~  U_i=\id_{\{X_i\le \beta\}}Z+\id_{\{X_i> \beta\}}Y_i,~~ i\in[n].
\end{align}
Clearly, $U_i\overset{\mathrm{d}}{\sim} \mathrm U(0,1)$ for all $i\in[n]$. Moreover, $(U_1,\dots,U_n)$ defined by  \eqref{ex:1}  is positively dependent as for $i\neq j$, with $\mathrm {corr}(U_i,U_j)=1-\beta^3$ if $(U_1,\dots,U_n)$ is the first case  and $\mathrm {corr}(U_i,U_j)=\beta^4$ if $(U_1,\dots,U_n)$ is  the second case. If $(U_1,\dots,U_n)$ is modelled by either case in \eqref{ex:1}, it is known that $S(U_1,\dots,U_n)\simeq_{\mathrm st} U_1$ where $S(u_1,\dots,u_n)= \min\left\{ n u_{(i)}/i\right\}$, $(u_1,\dots,u_n) \in (0,\infty)^{n}$ (see Example 1 in \cite{samuel1996simes} and Proposition 3.4 in \cite{XH22}). Let  $r\le -1$.
By Theorem 16 in \cite{HLP34} and Theorem 3 in \cite{CLTW21}, 
$M_r\le M_{-1}\le S$. Hence, $M_{r}(U_1,\dots,U_n)$ is  sub-uniform if $(U_1,\dots,U_n)$ is modelled by \eqref{ex:1}.
\end{example}

\begin{remark}
As the problem of merging p-values is closely related to the problem of risk aggregation (see Section \ref{sec:0}),  our results on \eqref{eq:maineq} under various conditions have a direct interpretation in risk management. For all $i\in[n]$, let $X_i=U_i^{r}$ with $r\le -1$. Then $X_i$ follows a Pareto distribution with infinite mean, which is widely used in modeling extremely heavy-tailed financial losses. For $(w_1,\dots,w_n)\in\Delta_n$ and $x>0$,  sub-uniformity of $U_1,\dots,U_n$ implies 
$\p(\sum_{i=1}^nw_iX_i>x)\ge \p(X_1>x)$. Therefore,  for identically distributed (but possibly dependent) infinite-mean Pareto losses,   the average loss is strictly riskier than any individual loss; that is, diversification is harmful. Putting this in the language of risk measures,  the Value-at-Risk,  which is a quantile of a loss position, is strictly superadditive  at any level for these losses. This issue was extensively discussed in \cite{CEW24}, and here we use more general dependence models. See also \cite{NEC06} for  the diversification of infinite-mean models in the context of operational risk. 
\end{remark}

\section{Threshold under independence}\label{sec:5}

As we have seen above, the harmonic mean p-value is anti-conservative under a wide range of dependence assumptions. It is then worth studying its threshold by which the type-I error rate can be properly controlled below the significance level. However, explicit expressions of the probability distributions of the harmonic mean p-value are generally not available, even when p-values are independent. In this section, we focus on the independence case and use the generalized central limit theorem to derive the asymptotic threshold of the harmonic mean p-value as the number of p-values goes to infinity. The asymptotic threshold is potentially useful in genomewide study where there are a large number of p-values (e.g., \cite{storey2003statistical}), although the independence assumption may not be verifiable in such a context.

For $\alpha\in(0,1]$, $q_\alpha(X)$ is the left $\alpha$-quantile of a random variable $X$, defined as
\begin{align*}
    q_{\alpha}(X)&=\inf\{x\in \mathbb{R}\mid \mathbb{P}(X\leq x)\ge \alpha\}.
\end{align*}
We also use $F^{-1}(\alpha)$ for $q_\alpha(X)$ if $X$ follows a distribution $F$. Let $U_1,\dots,U_n$ be independent. For $p\in(0,1)$,  denote  by $a_{n,p}$ the threshold of the symmetric harmonic mean of p-values, that is,
$a_{n,p}=q_p(M_{-1} (U_1,\dots,U_n)).$ It is clear that
\begin{equation*}
\p(M_{-1} (U_1,\dots,U_n) < a_{n,p}) \le p.
\end{equation*}

  Let $S_1$ be a distribution function with characteristic function given by 
\begin{equation*}
    \int_\R \exp(i\theta x)\d S_1(x) =    \exp\left(-|\theta|(1+i\frac{2}{\pi}\sign(\theta)\log|\theta|)\right)~~\mbox{for}~~\theta \in \R,
\end{equation*}
where $\sign(\cdot) $ is the sign function. The distribution $S_1$ is a stable distribution with tail parameter $1$ (see \cite{samoradnitsky2017stable}). The following proposition gives an asymptotic approximation of $a_{n,p}$ for large $n$. Using $a_{n,p}$ is equivalent to the asymptotically exact test of the harmonic mean p-value method of \cite{W19}.
\begin{proposition}\label{prop:threshold}
For  $p\in (0,1)$, let $a_{n,p}$ be the threshold of $M_{-1}$. Then   
\begin{equation}\label{eq:threshold}
a_{n,p}\sim \(\frac{\pi}{2}S_{1}^{-1}(1-p)+\log\(\frac{n\pi}{2}\)+1-\gamma\)^{-1} \rightarrow 0 \mbox{~~~as $n \rightarrow \infty$},
\end{equation}
where $\gamma$ is the the Euler--Mascheroni constant.\footnote{The Euler--Mascheroni constant $\gamma$ is approximately $0.57721$.}
\end{proposition}
\begin{proof}
Note that the random variables  $U_{1}^{-1},\dots,U_{n}^{-1}$ follow a Pareto distribution with distribution function $\mathbb{P}(U^{-1}\leq x)=1-x^{-1}$, $x\in[1,\infty)$.
By the generalized central limit theorem (see Theorem 1.8.1 in \cite{samoradnitsky2017stable}), sum of iid Pareto random variables $U_{1}^{-1},\dots,U_{n}^{-1}$ behaves like a stable distribution with tail parameter $1$ for large  $n$. Let  $Z\overset{\mathrm{d}}{\sim} S_1$. Hence for $p\in(0,1)$, 
\begin{align*}
    \lim_{n\rightarrow \infty}\frac{\p\(M_{-1} (U_1,\dots,U_n) \le p\)}{\p\(Z\ge c_n^{-1}(np^{-1}-d_n)\)}=\lim_{n\rightarrow \infty}\frac{\p\(c_n^{-1}(\sum_{i=1}^nU_i^{-1}-d_n)\ge c_n^{-1}(np^{-1}-d_n)\)}{\p\(Z\ge c_n^{-1}(np^{-1}-d_n)\)}=1,
\end{align*}
where $d_n=\displaystyle \frac{\pi n^{2}}{2}\int_{1}^{\infty}\sin\left(\frac{2 x}{n\pi}\right) x^{-2} \d x$ and $c_n=n\pi/2$.
This implies
$$  a_{n,p}=q_p(M_{-1} (U_1,\dots,U_n))\sim  \(\frac{\pi}{2}S_{1}^{-1}(1-p)+ \frac{n\pi }{2}\int_{1}^{\infty}\sin\left(\frac{2 x}{n\pi}\right)  x^{-2} \d x\)^{-1}:=b_{n,p} \mbox{~~~~as $n \rightarrow \infty$}.$$
By Taylor's expansion and properties of the cosine integral, we get   
\begin{align*}
\frac{n\pi  }{2}\int_{1}^{\infty}\sin\left(\frac{2 x}{n\pi}\right)  x^{-2} \d x&=\frac{n\pi  }{2}\sin\(\frac{2}{n\pi}\)+\int_{\frac{2}{n\pi}}^\infty\frac{\cos(y)}{y}\mathrm{d}y\\
&=1+\sum_{i=1}^{\infty}\frac{(-1)^i(2n^{-1}\pi^{-1})^{2i}}{(2i+1)!}-\(\gamma+\log\(\frac{2}{n\pi}\)-\int^{\frac{2}{n\pi}}_0\frac{1-\cos(y)}{y}\mathrm{d}y\)\\
&\sim \log\(\frac{n\pi}{2}\)+1-\gamma.
\end{align*}
 Hence, as $n \rightarrow \infty$,
 $$a_{n,p}\sim b_{n,p}\sim \(\frac{\pi}{2}S_{1}^{-1}(1-p)+\log\(\frac{n\pi}{2}\)+1-\gamma\)^{-1}.$$ 
 We have the desired result.
\end{proof}
Proposition \ref{prop:threshold} means that as more independent p-values are merged by $M_{-1}$, a smaller threshold needs to be used. In other words, there does not exist a constant multiplier which makes the harmonic mean p-value valid,  {consistent with \cite{W19}}. By \eqref{eq:threshold},  the   multiplier $\kappa_n$ such that $\p(M_{-1} (U_1,\dots,U_n) < p) \le \kappa_n p$  {grows at a rate of $\log n$}  as $n$ goes to infinity.
This is in sharp contrast to the dependence structure modelled by the Clayton copulas in Theorem \ref{th:Clayton}, where the correction does not  {go to infinity} as $n$ increases.

\cite{CLTW21} showed that the harmonic mean p-value method is closely related to two commonly used merging methods, the Cauchy combination \citep{LX20} and the Simes methods \citep{S86}. 
The Cauchy combination method uses $M_g$ (see Section \ref{sec:0}) to combine p-values, where $g$ is chosen to be the quantile function of the standard Cauchy distribution, i.e., $g(p)=\tan(\pi(p-0.5))$ for $p\in(0,1)$. The Simes method merges p-values via the Simes function $S$ defined in Section \ref{sec:2}.
In contrast to the harmonic mean p-value method, the Simes and the Cauchy combination methods always produce valid merged p-values for any number of independent p-values.
Hence, no correction is required for the merged p-values of the Simes and the Cauchy combination methods. This holds not only for independent p-values, but also for p-values modelled by a wide range of dependence structures. For instance, the Cauchy combination method is valid for the dependence structures considered by \cite{pillai2016unexpected}.
The Simes method is also conservative if the test statistics follow a multivariate normal distribution with nonnegative correlations \citep{sarkar1998some}. In this case, however, both the harmonic mean p-value and the Cauchy combination methods seem to be anti-conservative, based on numerical experiments; see Section \ref{sec:6} for the harmonic mean p-value and the simulation results in \cite{CLTW21} for the Cauchy combination method.  {See also \cite{rustamov2020kernel} and \cite{fang2021heavy} for comparisons between the harmonic mean and the Cauchy combination methods.}

\section{Discrete uniform random variables}\label{sec:4}

In this section, instead of considering standard uniform random variables, we study discrete uniform random variables $U_1^m,\dots,U_n^m$ on a finite set $\{1/m,\dots,m/m\}$ of $m$ equidistant   points. 
This setting concerns discrete p-values, which may be obtained from, for instance, binomial test and  conformal p-scores; see \cite{VGS05} and the more recent \cite{BCRLS21}.


We first note that for discretely distributed $U_1^m,\dots,U_n^m$, one cannot expect  $$ M^{\mathbf w}_{r} (U_1^m,\dots,U_n^m) \preceq_{\mathrm st} U_1^m $$ to hold for any $r \in \R$ and $\mathbf w\in \Delta^+_n$ unless $U_1^m,\dots,U_n^m$ are identical.
The reason is that $$
\p\left( M^{\mathbf w}_{r}(U_1^m,\dots,U_n^m) \le 1/m\right)  = 
\p\left(  U_1^m=\dots=U_n^m= 1/m\right) ,
$$
which is less than $1/m$ unless the events $U_i^m=1/m$ for $i\in[n]$ occur together almost surely.
 Applying similar arguments on $\p\left( M^{\mathbf w}_{r}(U_1^m,\dots,U_n^m) \le k/m\right)$, $k\in[m]$, leads to $U_i^m=k/m$ for $i\in[n]$ also occur together almost surely for all $k\in [m]$. Hence,  $ M^{\mathbf w}_{r} (U_1^m,\dots,U_n^m) \preceq_{\mathrm st} U_1 $ implies that  $U_1^m,\dots,U_n^m$ are identical.
This argument is similar to Proposition \ref{prop:trivial}.

In the context  of hypothesis testing, we are more interested in whether the following inequality holds,
\begin{equation}
\label{eq:discrete}
\p\left( M^{\mathbf w}_{r}(U_1^m,\dots,U_n^m) \le p \right)  > p ~~\mbox{for some pre-specified}~ p\in(0,1),
\end{equation}
where $r\le -1$.
 Based on previous discussions on  sub-uniformity for standard uniform random variables, we may expect that  \eqref{eq:discrete}  holds for  $(U_1^m,\dots,U_n^m)$ with large $m$ if 
 it has a copula for which sub-uniformity holds.
The intuition is that if $m$ is very large, the distribution of each $U_1^m,\dots,U_n^m$ is close to the uniform distribution on $(0,1)$. If $U_1^m,\dots,U_n^m$ are
NUOD,
we show below that \eqref{eq:discrete}  holds asymptotically as $m$ goes to infinity in the case of symmetric mean function. Following a similar line of thought,  the corresponding result also holds if $(U_1^m,\dots,U_n^m)$ has certain Clayton copula or an extremal mixture copula, as in Section \ref{sec:3}.

\begin{theorem}\label{th:discrete}
Let $r \le- 1$, $p\in(0,1)$, and  $U_1^m,\dots,U_n^m$ be  NUOD discrete uniform random variables on $\{1/m,\dots,m/m\}$, $m\ge2$. There exists a sequence $\{p_m:m\ge 2\}$ such that 
$$\p\(M_{r}(U_1^m,\dots,U_n^m)\le p\)\ge p_m\xrightarrow[m \to \infty]{} p.$$ 
Moreover, if $m>n^{-1/r}p^{-1}$, we can take
$$p_m=p-\frac{p^{1-r}}{m}\(\(np^r-(n-1)\(\frac{m+1}{m}\)^r\)^{1/r}-\frac{1}{m}\)^{r-1},$$
and $p_m=0$ otherwise.
\end{theorem}
\begin{proof}
By Theorem 3 of \cite{LWZZ24} and its proof, random variables $U_1^m,\dots,U_n^m$ are   NUOD if and only if 
there exist NUOD standard uniform random variables $V_1^m,\dots,V_n^m$ such that 
$$U_i^m=\sum_{j=1}^m \frac{j}{m}\id_{\{(j-1)/m<V_i^m\le j/m\}},~~~ i\in[n].$$ 
For $p\in(0,1)$,  let $R(p)=\p\(M_{r}(U_1^m,\dots,U_n^m)\le p\)$, $r\le -1$.  Define the following events 
$$A=\left\{\frac{1}{n}\sum_{i=1}^n\(V_i^m+\frac{1}{m}\)^r< p^r\right\},~~B=\left\{\frac{1}{n}\sum_{i=1}^n \(V_i^m\)^r\ge p^r\right\},$$
$$\mbox{~and~~}C=\left\{\frac{1}{n}\sum_{i=1}^n\(\(V_i^m\)^r+\frac{r}{m}\(V_i^m\)^{r-1}\)< p^r\right\}.$$
Note that $U_i^m\le V_i^m+1/m$ for all $i\in[n]$.  We have 
\begin{align}\label{eq:th3-1}
R(p)
&=\p\(\frac{1}{n}\sum_{i=1}^n(U_i^m)^r\ge p^r\) \nonumber\\
&\ge\p\(\frac{1}{n}\sum_{i=1}^n \(V_i^m+\frac{1}{m}\)^r\ge p^r\)\nonumber \\
&\ge\p\(\left\{\frac{1}{n}\sum_{i=1}^n \(V_i^m+\frac{1}{m}\)^r\ge p^r\right\}\cap B\)\nonumber\\
&=\p\(B\)-\p\(A\cap B\) =\p\(B\)-\p\(A\cap B\cap C\).
\end{align}
The last equality is due to  $A\subseteq C$  as for fixed $x\in (0,1)$, $(x+\epsilon)^r\ge x^r+\epsilon rx^{r-1}$, $\epsilon\in(0,1)$.
 Note that
\begin{align}\label{eq:th3-2}
A
&=
\bigcap_{i=1}^n\bigg\{
\(V_i^m+\frac{1}{m}\)^r< np^r-\sum_{j\in[n]/i}\(V_j^m+\frac{1}{m}\)^r
\bigg\} \nonumber\\
&\subseteq
\bigcap_{i=1}^n\bigg\{
\(V_i^m+\frac{1}{m}\)^r< np^r-(n-1)\(\frac{m+1}{m}\)^r
\bigg\} \nonumber\\
&=
\bigcap_{i=1}^n\bigg\{
V_i^m> \(np^r-(n-1)\(\frac{m+1}{m}\)^r\)^{1/r}-\frac{1}{m}
\bigg\} =
\bigcap_{i=1}^n\bigg\{
V_i^m> z_p
\bigg\},
\end{align} 
  where 
 $$z_p=\(np^r-(n-1)\(\frac{m+1}{m}\)^r\)^{1/r}-\frac{1}{m}.$$
Note that $z_p$ is positive for large $m$ and negative for small $m$. If  $m\le n^{-1/r}p^{-1}$, we let $p_m=0$ to get the trivial bound $R(p)\ge 0$.
  
  We next focus on the case that $m>n^{-1/r}p^{-1}$. Since $p>n^{-1/r}m^{-1}\ge(n^{-1}(m^{-r}+(n-1)(m+1)^r m^{-r}))^{1/r}>m^{-1}$, we can verify that $0<z_p<1$. Let 
  $$D=\bigcap_{i=1}^n\{
V_i^m> z_p
\}.$$
  By \eqref{eq:th3-1} and \eqref{eq:th3-2}, we have
\begin{align*}
R(p)
&\ge\p(B)-\p(A\cap B\cap C)\\
&=\p(B)-\p(A\cap B\cap C\cap D)\\
&\ge \p(B)-\p(B\cap C\cap D)\\
&= \p\(B\)-\p\(\left\{p^r\le\frac{1}{n}\sum_{i=1}^n \(V_i^m\)^r< p^r-   \frac{1}{n}\sum_{i=1}^n \frac{r}{m}\(V_i^m\)^{r-1}\right\}\cap D\)\\
&\ge \p\(B\)-\p\(\left\{p^r\le\frac{1}{n}\sum_{i=1}^n \(V_i^m\)^r< p^r-\frac{rz_p^{r-1}}{m}\right\}\cap D\)\\
&\ge\p\(\frac{1}{n}\sum_{i=1}^n \(V_i^m\)^r\ge p^r\)-\p\(p^r\le\frac{1}{n}\sum_{i=1}^n \(V_i^m\)^r< p^r-\frac{rz_p^{r-1}}{m}\)\\
&=\p\(M_{r}(V_1^m,\dots,V_n^m)\le \(p^r-\frac{rz_p^{r-1}}{m}\)^{1/r}\).
\end{align*}
By Proposition \ref{th:NLD},
\begin{align*}
R(p)&\ge \(p^r-\frac{rz_p^{r-1}}{m}\)^{1/r}\\
&\ge p-\frac{z_p^{r-1}}{m}p^{1-r}\\
&=p-\frac{p^{1-r}}{m}\(\(np^r-(n-1)\(\frac{m+1}{m}\)^r\)^{1/r}-\frac{1}{m}\)^{r-1}=p_m. 
\end{align*}
It is straightforward to verify that $p_m$ goes to $p$ as $m$ goes to $\infty$. 
\end{proof}

\section{Numerical examples}\label{sec:6}
Throughout this section, let $R_{n}(p)=\p(M_{-1}(U_1,\dots,U_n)\le p)$ for $p\in (0,1)$, where $U_1,\dots,U_n$ are standard uniform random variables or discrete uniform random variables on $\{1/m,\dots,m/m\}$ with $m\ge 2$.
  We first provide a few small numerical examples to illustrate sub-uniformity for dependent $U_1,\dots,U_n$. 
 The first example is for standard uniform random variables, which follow  the copula  generated by an equicorrelated Gaussian distribution with  $\rho\in[0,1]$.  Let $\Phi$ be the standard normal distribution function, and $Z,Z_1,\dots,Z_n$ be independent identically distributed standard normal random variables. Write 
 $$U_{i}=\Phi(X_{i}),\mbox{~where~}X_{i}=\rho Z+\sqrt{1-\rho^{2}}Z_{i}, \mbox{~$i\in[n]$~}.$$
 Fix $p=0.1$.  In Figure \ref{R1}, we display $R_{n}(p)$ for $n=5,10,15,20$, and $\rho\in[0,1]$. We observe that  sub-uniformity holds for all $\rho\in[0,1]$, and that as $n$ increases, $R_n(p)$ gets larger. These results show that  sub-uniformity may also hold for the class of equicorrelated Gaussian copulas with positive correlations, but the results in this paper can only cover the case of Gaussian copulas with non-positive correlations due to negative upper orthant dependence, and the corresponding sub-uniformity statement for a general  positive $\rho$ is not known in the literature. If the significance level $p$ is extremely small, it is expected that the inflation issue will be less severe, as justified by Theorem 2 of \cite{CLTW21}.
 \begin{figure}[h]
\centering
\includegraphics[height=3.75cm ,trim={0 10 0 10},clip]{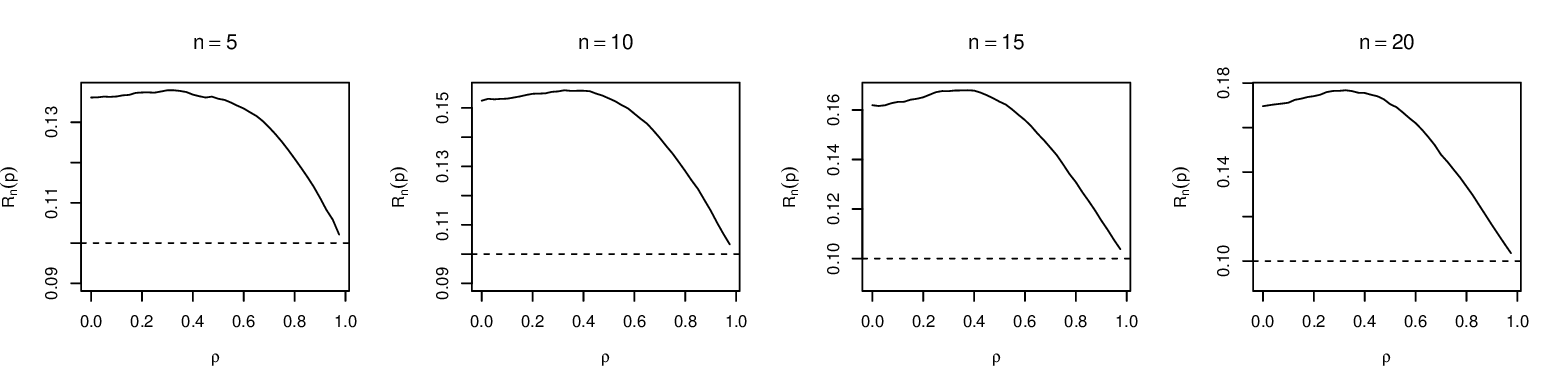}
\caption{Equicorrelated Gaussian copula: $R_{n}(p)$ for $n=5,10,15,20$, and $\rho\in[0,1]$ with $p=0.1$.} 
\label{R1}
\end{figure}
 
In the second example, we consider the case of the Clayton copula. Assume that $p=0.1$ and
$(U_1,\dots,U_n)\overset{\mathrm{d}}{\sim}\mathrm{Clayton}(t)$ where $t\in(0,1.5)$. From Figure \ref{fig:clayton}, we can see that $R_n(p)$ decreases as $t$ increases (i.e., dependence gets stronger). Moreover, for $t\ge 1$, $R_n(p)$ is only slightly larger than $p$, consistent with Theorem \ref{th:Clayton}. On the other hand, as $t$ approaches from 1 to 0,  $R_n(p)$ will increase rapidly as the dependence structure becomes closer to independence. 
 \begin{figure}[h]
\centering
\includegraphics[height=3.75cm ,trim={0 10 0 10},clip]{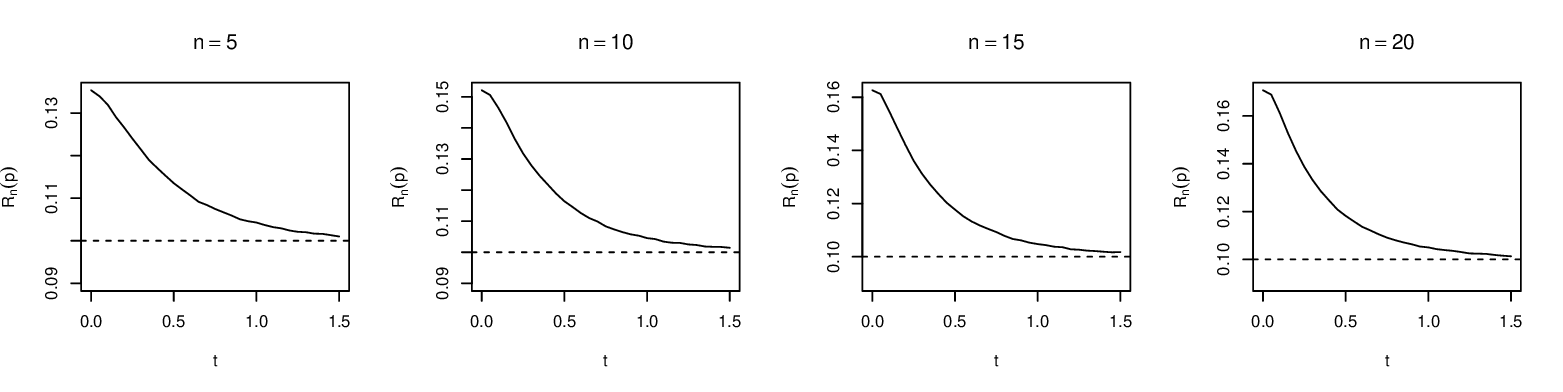}
\caption{Clayton copula: $R_{n}(p)$ for $n=5,10,15,20$, and $t\in[0,1.5]$ with $p=0.1$.} 
\label{fig:clayton}
\end{figure}

 As sub-uniformity holds for some Gaussian copulas and Clayton copulas, a natural question is whether it also holds for the more general classes of elliptical copulas or Archimedean copulas.  A theoretical answer to this question is not available with the current techniques and we present below some simulation studies. We first compute $R_n(0.1)$ by assuming that $(U_1,\dots,U_n)$ follows an equicorrelated  t-copula with degrees of freedom $4$ and correlation coefficient $\rho\in[0,1]$.  Figure \ref{fig:tcopula} suggests that sub-uniformity may hold for t-copulas with positive correlation coefficients and $R_n(p)$ decreases as the positive dependence becomes stronger.
  \begin{figure}[h]
\centering
\includegraphics[height=3.75cm ,trim={0 10 0 10},clip]{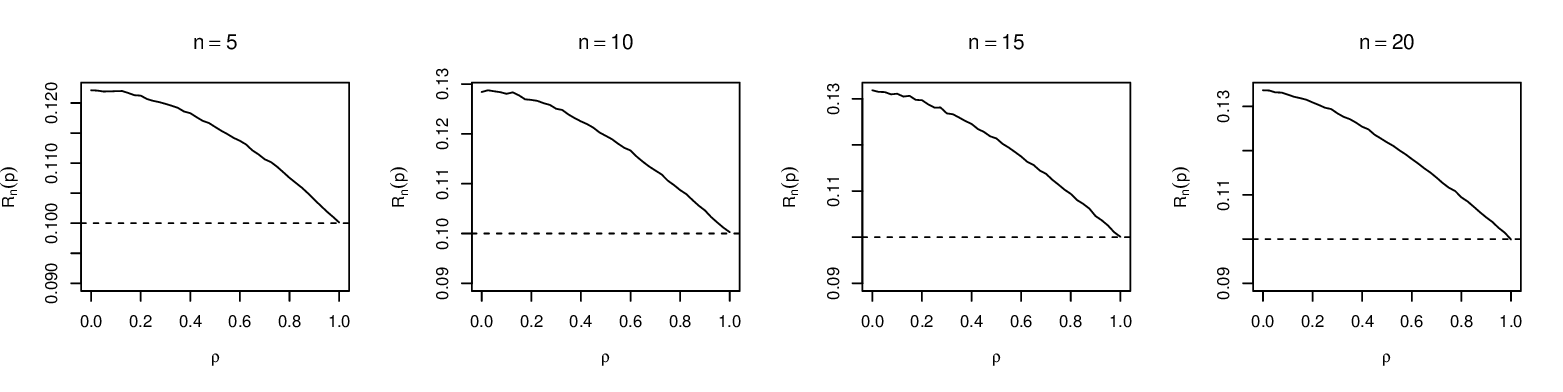}
\caption{Equicorrelated t-copula: $R_{n}(p)$ for $n=5,10,15,20$, and $\rho\in[0,1]$ with $p=0.1$.} 
\label{fig:tcopula}
\end{figure}
Next, we compute $R_n(p)$ for the Gumbel copula, which falls in the class of  Archimedean copulas. The Gumbel copula with parameter $\theta\in(1,\infty)$ is defined as 
$$
C^G(u_1,\dots,u_n) = \exp\(-\((-\log u_1)^\theta+\dots+(-\log u_2)^\theta\)^{1/\theta}\),~~ (u_1,\dots,u_n)\in(0,1)^n.
$$
If $\theta=1$, $C^G$ is the independence copula. As $\theta$ goes to $\infty$, $C^G$ approaches the comonotonicity copula. Figure \ref{fig:gumbel} plots $R_n(0.1)$ against $\theta\in(1,10)$. The observation is similar to the case of the Clayton copula: $R_{n}(p)$ is large when the Gumbel copula is close to the independence copula and is small when the Gumbel copula is close to the comonotonicity copula. Based on these numerical results, we conjecture that sub-uniformity holds for more general elliptical copulas and Archimedean copulas than the ones considered in this paper. Nevertheless, we should keep in mind that sub-uniformity requires $R_n(p)\ge p$ for all $p\in (0,1)$, and not only for a specific  $p$ as in our figures.
 \begin{figure}[h]
\centering
\includegraphics[height=3.75cm ,trim={0 10 0 10},clip]{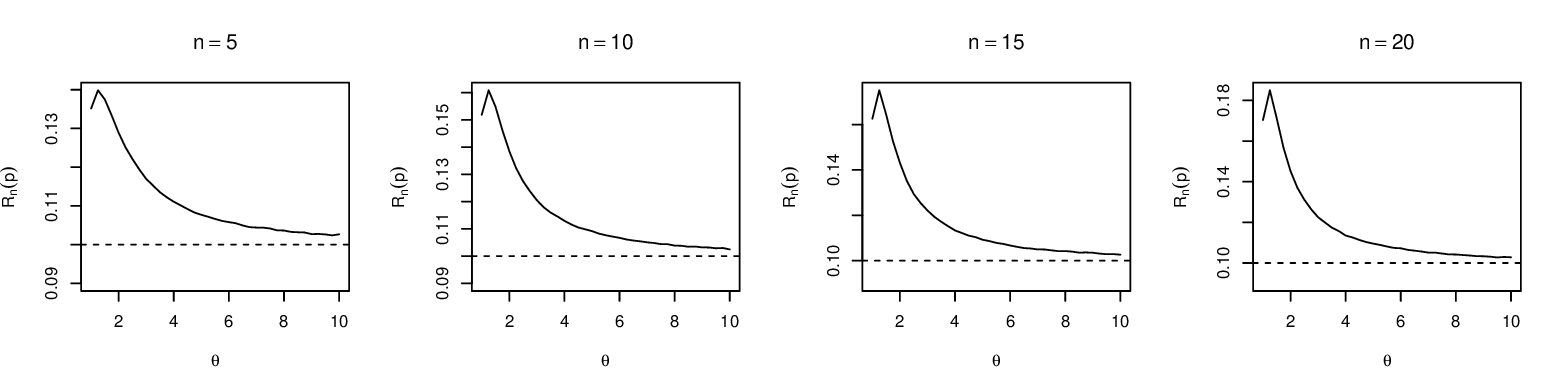}
\caption{Gumbel copula: $R_{n}(p)$ for $n=5,10,15,20$, and $\theta\in[1,10]$ with $p=0.1$.} 
\label{fig:gumbel}
\end{figure}

The fifth example presents the case of independent discrete uniform random variables on $\{1/m,\dots,m/m\}$, $m\ge 2$. Figure \ref{discrete} gives $R_n(p)$ for $10$  discrete uniform random variables with different discretization $m$. We can see that as $m$ increases, $R_n(p)$ for discrete uniform random variables becomes closer to that for uniform random variables. Moreover, if $m$ is large, \eqref{eq:discrete} holds for a wide range of $p$ in $(0,1)$ except for extremely small ones. 

 \begin{figure}[h]
\centering
\includegraphics[height=3.7cm ,trim={0 10 0 10},clip]{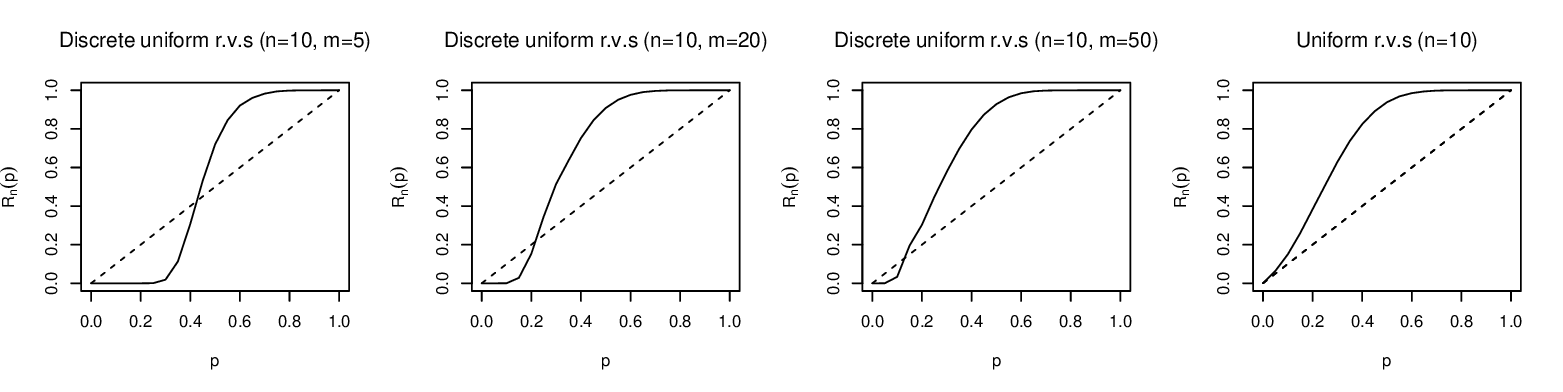}
\caption{$R_{n}(p)$ for discrete p-values ($m=5,20,50$) and uniform p-values, with $p\in(0,1)$.} 
\label{discrete}
\end{figure}

In Figure \ref{threshold}, we numerically compute the threshold $a_{n,p}$ of the harmonic mean p-value for independent p-values and its asymptotic form \eqref{eq:threshold}. The thresholds are computed at significance levels $0.01$, $0.05$, and $0.1$, up to $5000$ p-values. The results suggest that the asymptotic threshold \eqref{eq:threshold} can be a very good approximation of $a_{n,p}$ for large numbers of p-values. 
 As hinted by the plot for $p=0.01$,  the simulated results are not stable for small significance levels, e.g., 0.005, and are thus omitted here. 

 \begin{figure}[h]
\centering
\includegraphics[height=5cm ,trim={0 10 0 10},clip]{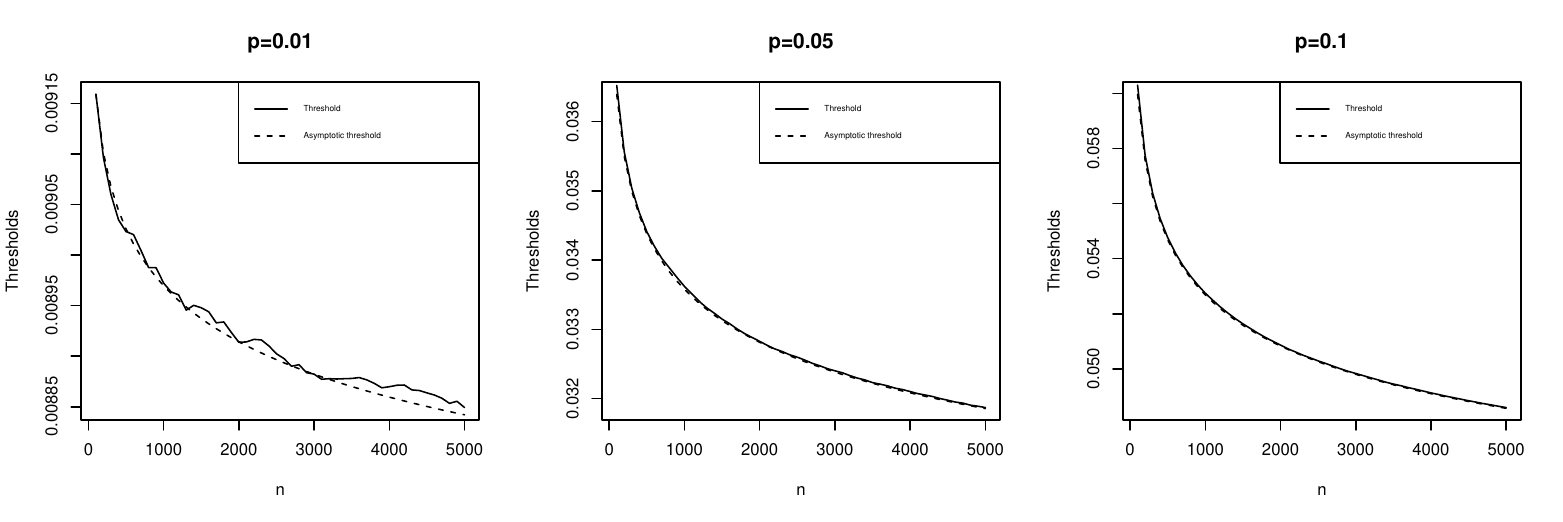}
\caption{The asymptotic threshold \eqref{eq:threshold} and $a_{n,p}$ of the harmonic mean p-value for independent p-values.} 
\label{threshold}
\end{figure}

  \section{Conclusion}\label{sec:7}
Sub-uniformity of generalized means of standard uniform random variables $U_1,\dots,U_n$ is studied under several dependence assumptions. In particular,  sub-uniformity is shown to hold in three cases: (i) negative upper orthant dependence  (Proposition \ref{th:NLD}); (ii) the class of extremal mixture copulas (Theorem \ref{prop:extremal}); (iii) some Clayton copulas (Theorem \ref{th:Clayton}). These dependence structures can be used to construct a wide range of dependence structures for which sub-uniformity holds, as suggested by Propositions \ref{prop:trivial2} and \ref{prop:trivial3}. Based on some numerical results, we conjecture that sub-uniformity also holds for Gaussian copulas with positive correlation coefficients.   An important implication of sub-uniformity in multiple hypothesis testing is that  merging p-values by any $r$-generalized mean function with $r\le -1$ is anti-conservative across all significance levels  in $(0,1) $. Although sub-uniformity  cannot hold for discrete uniform random variables,  using an $r$-generalized mean function with $r\le -1$ can still be anti-conservative if the number of discretizations is large (Theorem \ref{th:discrete}). An asymptotic threshold of the harmonic mean p-value for independent p-values is derived in Proposition \ref{prop:threshold}. As the number of p-values increases, since the asymptotic threshold goes to 0, the harmonic mean p-value will be more anti-conservative if no adjustment is applied.

For the purpose of multiple testing under dependence, due to the anti-conservativeness results found in this paper, we recommend using the Simes method  or the Cauchy combination, which are valid under independence and some other dependence assumptions, as well as their variants, over the harmonic mean p-value. The main advantage of the Simes and the Cauchy combination methods is that no correction is needed to make their merged p-values valid for a wide range of dependence assumptions and different numbers of p-values, thus making these methods robust in some sense. On the other hand, correction of the harmonic mean p-value is necessary and may vary according to different dependence structures and numbers of p-values. 
Theorem \ref{th:Clayton} also gives a small threshold correction for the harmonic mean p-value under  Clayton copulas, suggesting that the harmonic mean p-value may behave quite well under some forms of positive dependence.

 We close the paper by noting that, 
although sub-uniformity can hold under a wide range of dependence structures of $U_1,\dots,U_n$,   there always exists some dependence structure under which sub-uniformity does not hold. For instance, since comonotonicity (i.e., $U_1=\dots=U_n$ almost surely) does not maximize the distribution function of the sum of random variables (see \cite{WPY13} for bounds on the distribution function of the sum), it is always possible to construct a dependence structure of $U_1,\dots,U_n$ such that 
$\p(w_1U_1^{r}+\dots+w_nU_n^{r}\le t) >  \p(U_1^{r}\le t)$ for some threshold $t\in \R$ of interest. Therefore, conditions on dependence structures that lead to sub-uniformity or super-uniformity, other than the ones studied in this paper, require further research. This may be achieved by research on either the statistical problem of merging p-values, 
 the risk management problem of quantifying risk aggregation, or their interplay.

\end{document}